\documentclass[12pt]{article}

\usepackage{amsmath,amsfonts,amssymb}
\usepackage{amsthm}
\usepackage{mathtools}
\usepackage{verbatim}
\usepackage{multicol}
\usepackage{ bbold }
\usepackage{enumitem}
\usepackage{graphicx}
\usepackage[margin=1in]{geometry}
\usepackage{fancyhdr}
\usepackage[capitalize,noabbrev]{cleveref}
\usepackage{indentfirst}
\usepackage{sectsty}

\sectionfont{\large}

\makeatletter
\def\thmhead@plain#1#2#3{%
	\thmname{#1}\thmnumber{\@ifnotempty{#1}{ }\@upn{#2}}%
	\thmnote{ {\the\thm@notefont#3}}}
\let\thmhead\thmhead@plain
\makeatother

\makeatletter
\newcommand{\tpitchfork}{%
	\vbox{
		\baselineskip\z@skip
		\lineskip-.52ex
		\lineskiplimit\maxdimen
		\m@th
		\ialign{##\crcr\hidewidth\smash{$-$}\hidewidth\crcr$\pitchfork$\crcr}
	}%
}
\makeatother

\newtheorem{theorem}{Theorem}
\newtheorem*{theorem*}{Theorem}

\newtheorem{lemma}{Lemma}

\newtheorem{definition}{Definition}

\newcommand{\ZZ}{\mathbb{Z}}
\newcommand{\RR}{\mathbb{R}}

\newcommand{\HH}{\mathbb{H}}

\newcommand{\eps}{\mathbb{\varepsilon}}

\title{Sparse maps}
\author{Elia Portnoy}
\date{}

\begin{document}
	\maketitle

	\begin{abstract} 
		We prove a generalization of the Kolmogorov-Barzdin theorem for maps from simplicial complexes into Euclidean space. Along the way we introduce the notion of sparse maps and discuss maps from simplicial complexes with controlled 1-waist.
	\end{abstract}
	
	\section*{\center Introduction}
	
	This paper is concerned with the following quantitative embedding problem. Suppose $Y$ is a simplicial complex with $V$ vertices and with a small number of simplicies adjacent to each vertex. What is the smallest number $R$, so that $Y$ admits a map into a ball of radius $R$ in $\RR^n$, for which the pre-image of each unit ball intersects a small number of simplices from $Y$? To make $R$ small we might have to stretch and fold $Y$ in some complicated way. In \cite{KB}, Kolmogorov and Barzdin proved the first such result about maps from graphs into $\RR^3$. They showed that for any graph $Y$ with $V$ vertices and degree at most $D$, there is an embedding of $Y$ into a ball of radius $V^{1/2}$ in $\mathbb{R}^3$, so that the pre-image of each unit ball intersects at most $C(D)$ edges of $Y$. Here $C(D)$ is some constant that only depends on $D$. They also found that this bound on the radius is sharp up to a constant for families of expander graphs. In \cite{GG}, Gromov and Guth generalized this result to maps from simplicial complexes as follows. Suppose $Y$ is a $d$-dimensional simplicial complex with $V$ vertices, so that each vertex lies in at most $D$ simplices. Then for $n \ge 2d+1$, $Y$ can be embedded into a ball of radius $V^{\frac{1}{n-d}}log(V)^{2d+2}$ in $\mathbb{R}^n$, so that the distance in between the images of any two simplices which do not share a vertex is at least $c(n,D)$. Here, $c(n, D)$ stands for a constant that only depends on $n$ and $D$. In particular, this means that the pre-image of each ball of radius $c(n,D)$ in $\RR^n$ intersects at most $D$ simplices in $Y$.
	
	Our main theorem concerns maps which are $m$-sparse, which is a strengthening of the condition that the pre-image of every unit ball intersects a small number of simplices. In the next section we will give a detailed definition of $m$-sparse maps. Roughly, a map from a $d$-dimensional simplicial complex $Y$ to $\RR^n$ is $m$-sparse with overlap $S$ if it satisfies two conditions. First, the $k$-skeleton of $Y$ gets mapped into the the $k$-skeleton of the unit lattice in $\RR^n$, for each $0 \le k \le d$. Second, each $m$-dimensional coordinate plane in the unit lattice intersects the images of at most $S$ $d$-simplices from $Y$ in a strongly non-transverse way. For example, if $Y$ is a set of vertices, then a $1$-sparse map from $Y$ into $\RR^2$ is a way to place these vertices on $\ZZ^2 \subset \RR^2$ so that each vertical and horizontal line intersects at most $S$ vertices. So if $Y$ is a set of $V$ vertices, then there is a 1-sparse map with overlap 1 from $Y$ to $[0,V]^2 \subset \RR^2$, where we injectively map the vertices to the diagonal of $[0,V]^2$. If $Y$ is a graph, a 1-sparse map with overlap $S$ that maps $Y$ to $\RR^n$ will map each edge of $Y$ to a path in the unit grid so that the pre-image of each straight line in the unit grid intersects at most $S$ edges from $Y$ in a 1-dimensional set. Kolmogorov's and Barzdin's proof actually shows that if a graph $Y$ has $V$ vertices and maximum degree $D$, then there is a 1-sparse map with overlap $C(D)$ from $Y$ to a ball of radius $V^{1/2}$ in $\RR^3$.  The larger $m$ is, the more simplices can be placed inside an $m$-plane, and so more restrictions are placed on a map to be $m$-sparse. Notice, however, that the definition of an $m$-sparse map contains no direct restrictions on how many simplices can intersect an $m$-dimensional plane transversely. Here is our main theorem. 
	
	\begin{theorem} \label{sparse} Fix three dimensions $0 \le d \le m < n$ and an integer $D > 0$. Suppose $Y$ is a $d$-dimensional simplicial complex with $V$ vertices, so that each vertex lies in at most $D$ simplices. Then there is a map
		
		$$F: Y \to [0, V^{\frac{1}{n-m}}]^{n} \subset \mathbb{R}^{n}$$
		
		\noindent which is $m$-sparse map with overlap $S(n,D)$, where $S(n,D)$ only depends on $n$ and $D$.
	\end{theorem}
	
	If $d = m$ we get a result that is similar to Gromov's and Guth's theorem, but without the polylog factor. Namely, any $d$-dimensional simplicial complex $Y$ with $V$ vertices and at most $D$ simplices adjacent to any vertex, can be mapped into a cube of side-length $V^{\frac{1}{n-d}}$ in $\mathbb{R}^n$, so that the pre-image of each unit ball intersects at most a small number of simplices only depending on $n$ and $D$. Recall that in \cite{GG}, the authors assume that $2d+1 \le n$ and they impose the constraint that the distances between the images of any two disjoint simplices be at least some constant. We will show that a sparse map with $2d+1 \le n$, can be perturbed to also satisfy this constraint. So our notion of a $d$-sparse map can be seen as a generalization of the kind of maps considered in \cite{GG}. When $m > d$, \cref{sparse} constructs a map that I have not seen in the literature before. For example, if $d=1, m=2$ and $n=4$, this theorem says that we can wire any graph with $V$ vertices and maximal degree $D$ through a 4-dimensional unit grid of side-length $V^{1/2}$, so that each coordinate 2-plane has at most $S(D)$ edges wired along it somewhere non-transversely.
	
	The proof of \cref{sparse} is modeled after Kolmogorov's and Barzdin's original proof, but involves some inductive arguments. Here is the main idea for the case $d=1$ when $Y$ is a graph with $V$ vertices. Suppose we have we already have found a $(m-1)$-sparse map $f$ from the vertices of $Y$ into $[0, V^{\frac{1}{n-m}}]^{n-1} \times \{0\} \subset \RR^n$, and would like to extend $f$ to a $m$-sparse map $F: Y \to [0, V^{\frac{1}{n-m}}]^{n} \subset \RR^n$. $F$ will map each edge $e$ in $Y$ to a piece-wise linear path in $\RR^n$. To describe this path, we will choose an integer label $L(e)$ between $1$ and $V^{\frac{1}{n-m}}$ for each edge $e$ in $Y$. If $v_1$ and $v_2$ are the two boundary points of $e$, then the path $F(e)$ will contain two initial segments going from $f(v_i)$ to $f(v_i) + L(e)u_n$, where $u_n$ is the $n^{th}$ unit coordinate vector in $\RR^n$ and $i \in \{1,2\}$. To complete $F(e)$ we just need to connect the endpoints of these two segments with a path inside $[0, V^{\frac{1}{n-m}}]^{n-1} \times \{L(e)\}$. So after we add in these initial segments for each edge of $Y$, we reduce our original problem to constructing a $m$-sparse set of paths within each set $[0, V^{\frac{1}{n-m}}]^{n-1} \times \{j\}$, for $1 \le j \le V^{\frac{1}{n-m}}$. If we label our edges carefully, then the number of paths we need to construct in each hyperplane is an order of magnitude smaller than $V$. In the proof of \cref{sparse}, we will exploit such a reduction to set up a more general inductive argument. 
	
	In the next section we will explain why \cref{sparse} is sharp for simplicial complexes with large waists, called high dimensional expanders. Our next theorem offers one possible improvement on \cref{sparse} for simplicial complexes with a small waist. Define the 1-waist of a $d$-dimensional simplicial complex $Y$ as follows
	
	$$W_1(Y) =  \min_{f:Y \to \RR} \max_{p \in \RR} |f^{-1}(p)|$$
	
	\noindent where the min is over all continuous functions, and $|f^{-1}(p)|$ stands for the number of $d$-simplices that intersect the fiber $f^{-1}(p)$. For example, if a graph $Y$ has $V$ vertices and Cheeger constant $h$, then $W_1(Y)$ is bounded from below by roughly $hV$. So if $W_1(Y)$ is an order of magnitude smaller than $V$, then $Y$ is far from being an expander graph. Here is our next theorem.
	
	\begin{theorem}\label{smallwaist} Suppose $Y$ is a $d$-dimensional simplicial complex with $V$ vertices, so that each vertex lies in at most $D$ simplices. Then for $n \ge d+1$, there is a map
		
		$$F: Y \to [0,V^{\frac{1}{n}} W_1(Y)^{\frac{d}{n(n-d)}}]^n \subset \mathbb{R}^n$$
		
		\noindent so that the pre-image of each unit ball intersects at most $C(n, D)$ simplices, where $C(n,D)$ only depends on $n$ and $D$.
	\end{theorem} 
	
	\noindent The idea of the proof is to use the map which realizes the 1-waist to chop up $Y$ into pieces of size roughly $W_1(Y)$. Then we find a $d$-sparse map of each piece into a ball in $\RR^n$ using \cref{sparse}. We then arrange these balls in a row in $\RR^n$, map in the simplices that connect adjacent pieces to adjacent balls, and finally bend this row of balls into a larger ball in $\RR^n$. Note that if $W_1(Y)$ is roughly like $V$ then the bound in \cref{smallwaist} matches the one in \cref{sparse} for $d=m$.
	
	Here is an application of this theorem. Let $B_{\HH}$ be a ball of volume $V$ in hyperbolic plane. The area of a ball in hyperbolic plane is exponential in its radius, so $W_1(B_{\HH})$ is at most $polylog(V)$ (in fact, a result from \cite{H} shows that this waist is at most $log^2(1+V)$ times some constant). Therefore, \cref{smallwaist} says that there is a map from $B_{\HH}$ to a ball of volume $V polylog(V)$ in $\RR^{3}$, so that the pre-image of each unit ball in $\RR^{3}$ can be covered by a small constant number of unit balls in the hyperbolic plane. On one hand, such a map should expand many directions in $B_{\HH}$ by a factor of at least $V^{1/3} polylog^{-1}(V)$. On the other hand, the image of most unit balls in $B_{\HH}$ should be covered by about $polylog(V)$ balls in $\RR^3$. So this map must be very distorted. Based on this example the following question also seems interesting. Let  $B_{\HH^k}$ be a ball of volume $V$ in $k$-dimensional hyperbolic space. What is the smallest $R>0$, so that there is a map from $B_{\HH^k}$ into $[0,R]^n \subset \RR^n$, for which the pre-image of every unit ball can be covered by a small constant number of unit balls in $B_{\HH^k}$?
	
	Finally, we mention that another nice generalization of \cite{KB} and \cite{GG} is found in \cite{BH} and \cite{H}. There, the authors consider maps from simplicial complexes into symmetric spaces and products of trees, with a condition that is similar to having the pre-image of each unit ball be small. For example, one theorem there says that given any $d$-dimensional simplicial complex with $V$ simplices, where each vertex adjacent to at most $D$ simplices, there is a map from it to a product of $(d+1)$ binary trees $T^{d+1}$ (subdivided to be a simplicial complex), so that the image intersects at most $CVlog^d(1+V)$ simplices and the pre-image of each simplex in $T^{d+1}$ intersects at most $C$ simplices, where $C$ is some constant depending only on $n$ and $D$. The waists of several specific spaces are also more closely considered in \cite{H} under the name of topological overlap.    
	
	\par This paper is organized as follows. Section 1 contains the definition of sparse maps and a proof of \cref{sparse}. At the end of section 1 there is a  short discussion of why \cref{sparse} is sharp for high dimensional expanders and how to perturb sparse maps to separate the images of disjoint simplices. Section 2 contains the proof of \cref{smallwaist}.
	
	\section*{\centering Acknowledgments}
	I would like to thank Ting-Chun Lin for discussing embeddings of simplicial complexes with me. Ting-Chun Lin also mentioned some other methods that could lead to an improvement of the result in \cite{GG}. I was also like to thank Larry Guth for several helpful discussions about these topics, and Fedya Manin for pointing out that sparse maps can be perturbed to satisfy the condition from \cite{GG}. Finally, I would like to thank the anonymous referees who shared several helpful comments about the paper.
	
	\section*{\centering The main theorem for sparse maps}
	
	We begin this section by giving the definition of a $m$-sparse map. Let $\Lambda^n$ be the unit lattice in $\RR^n$ equipped with the natural skeletal decomposition. The $0$-skeleton of $\Lambda^n$ is $\mathbb{Z}^n$, the $1$-skeleton consists of all the edges between neighboring points of $\ZZ^n$ in the $n$ coordinate directions, and so on. The $k$-skeleton of $\Lambda^n$ will be denoted $\Lambda^n(k)$. A $k$-plane will denote any $k$-dimensional affine subspace in $\RR^n$ for which some $(n-k)$-coordinates are fixed integers. For example, $\{(x_1, 3, x_2, x_3, 2): (x_1, x_2, x_3) \in \RR^3\}$ is a 3-plane and is a subset of $\Lambda^5(3)$. Next we state a type of intersection that we will consider in order to define sparse maps.
	
	\begin{definition} Fix three dimensions $0 \le d \le m \le n$. Let $F$ be a map from a $d$-simplex $\sigma$ into $\RR^n$ and $H$ be an $m$-plane. We say that $H$ has a full intersection with the image of $\sigma$ if there is some open ball $B$ in the interior of $\sigma$, such that $F(B)$ lies inside $H$. In this case, we write $H \cap_f F(\sigma)$.
	
	More generally, suppose $F$ maps a $d$-dimensional simplicial complex $Y$ into $\RR^n$. Then we say that $H$ has a full intersection with $F(Y)$, if there is some $d$-simplex $\sigma$ in $Y$ with $H \cap_f F(\sigma)$. In this case, we write $H \cap_f F(Y)$.
	
	\end{definition}
	
	For example, if $q$ is a point in a plane $H$ and $F$ maps all of $\sigma$ to $q$, then $H$ has a full intersection with $F(\sigma)$.
	
	\begin{definition} 
		
		\par Fix three dimensions $0 \le d \le m \le n$ and an integer $S > 0$. Let $Y$ be a $d$-dimensional simplicial complex and denote the $k$-skeleton of $Y$ by $Y(k)$. A map $F: Y \to \RR^n$ is $m$-sparse with overlap $S$ if the following conditions are satisfied
		
		\begin{enumerate}
			\item For each dimension $0 \le k \le d$, we have $F(Y(k)) \subset \Lambda^n(k)$.
			\item For each $d$-simplex $\sigma$ of $Y$, there are at most $S$ $m$-planes which have a full intersection with $F(\sigma)$.
			\item For each $m$-plane $H$, there are at most $S$ $d$-simplices of $Y$ whose images have a full intersection with $H$.
		\end{enumerate}
	\end{definition}
	
	\noindent The first condition allows us to construct sparse maps combinatorialy, since their images are restricted to lie in a unit lattice. The second condition tells us that sparse maps are not too complicated when restricted to each simplex. The third condition restricts the number of simplices whose images are mapped along a single $m$-plane somewhere. Note, that the images of many simplices can cut across an $m$-plane transversely. This condition also implies that if $F$ is $m$-sparse then the pre-image of any unit ball intersects at most $C(n,D,S)$, simplices of $Y$, where $C(n,D,S)$ only depends on $n, D$ and $S$. This is because for any unit ball $B$ in $\RR^n$ and any $d$-simplex $\sigma$ of $Y$, whose image intersects $B$, there is an $m$-plane which intersects $B$ and has a full intersection with $\sigma$. However, there are at most ${n \choose m}$ $m$-planes that intersect any unit ball and each such $m$-plane has a full intersection with the images of at most $S$ $d$-simplices.
	
	Here is our main theorem about $m$-sparse maps.
	
	\begin{theorem*}[1] Fix three dimensions $0 \le d \le m < n$ and an integer $D > 0$. Suppose $Y$ is a $d$-dimensional simplicial complex with $V$ vertices, so that each vertex lies in at most $D$ simplices. Then there is a map
		
		$$F: Y \to [0, V^{\frac{1}{n-m}}]^{n} \subset \RR^{n}$$
		
		\noindent which is $m$-sparse with overlap $S(n,D)$, which only depends on $n$ and $D$.\footnote{Almost all constants in this paper will depend only on $n$ and $D$, but not on the size of the given simplicial complex. We include them to keep track of how the constants grow during the induction, but recommend ignoring them at a first reading.}
	\end{theorem*}
	
	\cref{sparse} will be proved using an inductive argument that involves $d,m$ and $n$. The base case will be $d=0$ and is given by the following lemma.
	
	\begin{lemma} \label{sparsevertex} Fix two dimensions $0 \le m < n$ and let $Y$ be a set of $V$ vertices. Then there is a map
		$$F: Y \to [0, V^{\frac{1}{n-m}}]^n \subset \RR^n$$
		\noindent which is $m$-sparse with overlap $2^n$.
	\end{lemma}
	
	\begin{proof} Note that for any $m$-plane $H$ and any point $v \in Y$ such that $F(v) \in H$, we have $H \cap_f F(v)$. We will show that such a $m$-sparse map $F$ exists by using a greedy algorithm. Order the vertices of $Y$ from $1$ to $V$. For some $1 \le k < V$, suppose we have found a way of mapping the first $(k-1)$ vertices of $Y$ into $\ZZ^n \cap [0, V^{\frac{1}{n-m}}]^n$ so that each $m$-plane contains at most $2^n$ vertices of $Y$. To find a place for the $k^{th}$ vertex, let us count the following set in two ways
	
	\begin{gather}
	Q = \{(\Pi, v, w):  \text{$\Pi$ is an $m$-plane, $v$ is one of the first $(k-1)$ vertices of $Y$}, \\
	v \in \Pi, w \in \Pi \cap \ZZ^n \cap  [0, V^{\frac{1}{n-m}}]^n\}
	\end{gather}
	
	\noindent The number of $m$-planes that contain one of the first $(k-1)$ vertices of $Y$ is at most ${n \choose m} (k-1)$. And the number of vertices in an $m$-plane in our domain is $V^{\frac{m}{n-m}}$ so we have
	
	$$|Q| \le (k-1) {n \choose m} V^{\frac{m}{n-m}} < 2^n V^{\frac{n}{n-m}}$$
		
	Now consider the function $\lambda: \ZZ^n \to \ZZ_{\ge 0}$, where $\lambda(p)$ counts how many of the first $(k-1)$ vertices of $Y$ are in some $m$-plane which contains $p$. We see that
		
	$$|Q| \ge \sum_{p \in \ZZ^n \cap [0, V^{\frac{1}{n-m}}]^n} \lambda(p)$$
		
	\noindent where the inequality can be strict if some $v$ and $w$, from the definition of $Q$, lie in several $m$-planes. So there must be some point $p \in \ZZ^n \cap [0, V^{\frac{1}{n-m}}]^n$ with $\lambda(p) < 2^n$. In other words, $p$ does not lie in an $m$-plane which contains $2^n$ of the first $(k-1)$ vertices of $Y$. Then we let $F$ map the $k^{th}$ vertex of $Y$ to this $p$. Continuing in this way we can construct the desired map $F$ which is $m$-sparse with overlap $2^n$.
	\end{proof}
	
	Our inductive argument will also involve a parameter called tightness which we now introduce.
	
	\begin{definition}  Let $\{\sigma_1, \sigma_2 \ldots \sigma_N\}$ be a set of $d$-simplices and set $X = \bigsqcup_{i=1}^N \partial \sigma_i$.
	Also fix a non-negative integer $q$ and a $q$-plane $\Pi$. A map $f: X \to \RR^n$ is $q$-tight with respect to $\Pi$, if for each $1 \le i \le N$, $f(\partial \sigma_i)$ is contained in some $q$-plane which is a parallel translate of $\Pi$. 
	\end{definition}
	
	For example, if for each $i$, $f$ maps $\partial \sigma_i$ into the $(n-1)$-plane $\RR^{n-1} \times \{i\}$ then $f$ is $(n-1)$-tight with respect to the $(n-1)$-plane $\RR^{n-1} \times \{0\}$. Next we state a lemma which tells us how to extend a map so that its tightness decreases, but its sparseness does not. This lemma takes a bit of time to state, but it is the crux of the argument. We also make a technical remark to avoid adding unnecessary notation. In the lemma we will deal with polyhedral complexes of the form $X \times I$, where $X$ is some simplicial complex and $I$ is an interval. Here, the polyhedra are of the form $\sigma \times I$ for some simplex $\sigma$. We will talk about sparse maps from such complexes and the definition of sparse maps generalizes in a straightforward way to polyhedral complexes.
	
	\begin{lemma} \label{tight} Fix  dimensions $0 \le d \le m < n$, $d \le q \le n-1$, and an integer $S > 0$. Let $\{\sigma_1, \sigma_2 \ldots \sigma_N\}$ be a set of $d$-simplices and set $X = \bigsqcup_{i=1}^N \partial \sigma_i$. Suppose for some $R > 0$ we are given a map
		
		$$g: X \to [0,R]^{n-1} \times \{0\} \subset \RR^n$$
		
		\noindent which is $(m-1)$-sparse map with overlap $S$. Furthermore, suppose that $g$ is $q$-tight with respect to some $q$-plane $\Pi \subset \RR^{n-1} \times \{0\}$ and suppose that $g(\partial \sigma_i)$ does not lie in an $(d-1)$-plane for $1 \le i \le N$. Then there is a map
		
		$$G: X \times [0,2] \to  [0, R]^{n} \subset \RR^n$$
		
		\noindent so that the following conditions hold
		
		\begin{enumerate}
			\item $G$ extends $g$ in the sense that, $G|_{X \times \{0\}} = g$
			\item $G(X \times \{2\}) \subset [0,R]^{n-2} \times \{0\} \times [0,R]$
			\item $G|_{X \times \{2\}}$ is $(q-1)$-tight with respect to a $(q-1)$-plane contained in $\RR^{n-2} \times \{0\} \times \RR$.
			\item $G|_{X \times \{2\}}$ is $(m-1)$-sparse with overlap $S'(n,S)$ as a map into $\RR^{n-2} \times \{0\} \times \RR$. Here $S'(n,S)$ is a constant which only depends on $n$ and $S$.
			\item $G$ is $m$-sparse with overlap $S'(n,S)$.
		\end{enumerate} 
		
	\end{lemma}

	Here is a rough picture of the lemma when $d=2, m=2$ and $n=3$. We start with a some loops $\{\partial \sigma_i\}$ mapped to the XY-plane plane inside $[0,R]^2 \times \{0\}$ with a $1$-sparse map $g$. Each loop will get a label $L(i)$, and will be extended in the Z-direction to a 2-dimensional cylindrical set $g(\partial \sigma_i) \times [0,L(i)]$. Here we slightly abuse notation and view $g(\partial \sigma_i)$ as a subset of $[0,R]^2$. The labels $L(i)$ will be chosen so that these sets are well spread out along the Z-direction. Then each such set will be further extended in the Y-direction, by adding in a cylindrical set which connects $g(\partial \sigma_i) \times {L(i)}$ to its linear projection in the Y-direction onto the line $[0,R] \times \{0\} \times \{L(i)\}$. The union of these cylindrical sets will be the images of $\{\partial \sigma_i \times [0,2]\}$ under the map $G$. Note that $G(\partial \sigma_i \times \{0\}) = g(\partial \sigma_i)$ lies in a plane, while $G(\partial \sigma_i \times \{2\})$ lies in a line. So the map $G$ can be viewed as a way of compressing the initial map $g$.

	\begin{proof} 
		\begin{comment}
			Suppose first that we are in the degenerate case $q \le d-1$. In this case, since $d \le m$, each parallel translate of $\Pi$ intersects the image of at most $S$ $d$-dimensional simplices of $X$ in a set of dimension $d$. So we get a $m$-sparse map $G$ by coning off each $\sigma_i$ to a point. To be more precise, let $p_i$ be a chosen point on $g(\partial \sigma_i)$ and for $x \in \sigma_i$ and $t \in [0,2]$, define
			
			$$G(x, t) = (1-t/2)x + (t/2)p_i$$
			
			Then $G(\partial \sigma_i \times [0,2])$ is $(d-1)$-dimensional and lies in the same $q$-plane $\Pi'$ that contains $g(\sigma_i)$. Since $q \le d-1$, to check that $G$ is $m$-sparse we only need to bound the pre-image of each unit ball. Suppose $B$ is a unit ball in $\RR^n$.
			
			Since $q \le d-1 \le m-1$, and $g$ is $(m-1)$-sparse, the number of simplices 	
			Next, suppose that $q \ge d$.
		\end{comment} 
		
		The idea of the proof is to spread out the images of all the pieces of $X \times [0,2]$ across parallel hyperplanes. By rotating our coordinate system, without loss of generality, we can assume that our tightness plane $\Pi$ contains the $(n-1)^{st}$ coordinate direction $\{(0, \ldots 0, t, 0) \in \RR^n: t \in \RR \}$. The $i^{th}$ simplex $\sigma_i$ will be assigned a label $L(i)$ and $G(\partial \sigma_i \times [1,2])$ will be contained in $\RR^{n-1} \times \{L(i)\}$. We begin by showing that there is a labeling of the simplices $\sigma_i$ with $R$ labels, given as
		
		$$L: \{1, 2, \ldots N\} \to \{1, 2 \ldots R\}$$
		
		\noindent so that for each $m$-plane $H$ and each label $j$, there are at most $S''=n^3S^2$ simplices $\sigma_i$ with $L(i) = j$ and with $H \cap_f g(\partial \sigma_i)$. In other words, $g$ is roughly $m$-sparse with overlap $S''$ when restricted to the boundaries of the simplices with label $j$.
		
		To find such a labeling we again rely on the greedy algorithm. Suppose for some $1 \le k < N$, we have already chosen suitable labels for $\sigma_1, \sigma_2 \ldots \sigma_{k-1}$. Let's see how to choose a suitable label for $\sigma_k$. Call a label $j$ bad, if there is a $m$-plane $H$ such that $H \cap_f g(\partial \sigma_k)$ and
		
		$$|\{\sigma_i: L(i) = j \text{ and } H \cap_f g(\partial \sigma_i), \text{ for } 1 \le i < k\}| \ge S''=n^3S^2$$
		
		\noindent Assume for contradiction that each label is bad. Consider the set
		
		$$Q = \{(H, \sigma_i, j):  L(i) = j \text{ and } H \cap_f g(\partial \sigma_i)  \text{ and } H \cap_f g(\partial \sigma_k), \text{ for } 1 \le i < k\}$$
		
		Let us count this set in two ways. Since we assumed that each label was bad and there are $R$ labels, we have $|Q| \ge RS''$. On the other hand, since $g$ is $(m-1)$-sparse, there are at most $(d+1)S$, $(m-1)$-planes that have a full intersection with $g(\partial \sigma_k)$, and each of these is contained in at most $n$, $m$-planes. So there are at most $(d+1)nS$, $m$-planes that have full intersection with $g(\partial \sigma_k)$. Each of these $m$-planes contains at most $mR$, $(m-1)$-planes that also intersect $[0,R]^n$. And, since $g$ is $(m-1)$-sparse, each of these $(m-1)$-planes have a full intersection with at most $S$ of the $\partial \sigma_i$. So counting $Q$ in a another way, we get 
		
		$$|Q| \le (d+1)nS(mR)S < Rn^3S^2 = RS''$$ 
		
		This contradicts our lower bound on $Q$, and we conclude that some label $j'$ is not bad. Then by setting $L(k)= j'$, $L$ becomes a suitable labeling of the first $k$ simplices. Proceeding in this way, we can define $L$ on all the simplices.
		
		We can now describe the map $G: X \times [0,2] \to [0,R]^n$. For each simplex $\sigma_i$, $G(\partial \sigma_i \times \{1\})$ will look like a projection of $g(\partial \sigma_i)$ to the hyperplane $\RR^{n-1} \times \{L(i)\}$. Also, $G(\partial \sigma_i \times \{2\})$ will look like a projection of $g(\partial \sigma_i)$ to the $(n-2)$-dimensional plane $\RR^{n-2} \times \{0\} \times \{L(i)\}$. In between these projections, $G$ will be defined by linear interpolation. More specifically, for each $i$ and each $x \in \partial \sigma_i$ with $g(x) = (g_1(x), g_2(x), \ldots g_{n-1}(x), 0)$, for $t \in [0,1]$ define
		
		$$G(x, t) = (g_1(x), g_2(x), \ldots g_{n-2}(x), g_{n-1}(x), tL(i))$$
		
		\noindent and for $t \in [1,2]$ define
		
		$$G(x, t) = (g_1(x), g_2(x), \ldots g_{n-2}(x), (2-t)g_{n-1}(x), L(i))$$
		
		Let's make some observations about the map $G$. First $G(x,0) = g(x)$ for each $x \in X$, so $G$ is in fact an extension of $g$. Since each label is at most $R$ we see that $G(X \times \{2\})$ lies in $[0,R]^{n-2} \times \{0\} \times [0,R]$. So the first two conditions on $G$ are satisfied. Next, recall that for each simplex $\sigma_i$, $g(\partial \sigma_i)$ is contained in a parallel translate of a $q$-plane which contains the $(n-1)^{st}$ coordinate direction. By the construction, $G(\partial \sigma_i \times \{2\})$ then lies in a parallel translate of a $(q-1)$-plane. This $(q-1)$-plane is the image of a projection of the plane $\Pi$ which decreases its dimension. In other words, $G|_{X \times \{2\}}$ is $(q-1)$-tight.
		
		Next, we consider the sparseness of $G$. Observe that if $B$ is a tiny ball in a $d$-dimensional polyhedron of $X \times [0,2]$, then $G(B)$ contains a segment parallel to the $n^{th}$ or $(n-1)^{st}$ coordinate direction. This means that if $H$ is an $m$-plane that contains $G(B)$, then it must contain such a segment as well. Let us say that a plane $H$ contains the $k^{th}$ direction, if it contains a line parallel to the $k^{th}$ coordinate direction. Suppose first that $H$ is an $m$-plane that contains the $n^{th}$ direction, and so intersects $\RR^{n-1} \times \{0\}$ in some $(m-1)$-plane $H'$. Suppose $\sigma$ is some $(d-1)$-dimensional simplex of $X$. If $G(\sigma \times [0,1]) \cap_f H$, then we see that $g(\sigma) \cap_f H'$ as well. If $G(\sigma \times [1,2]) \cap_f H$, then $H$ contains the $n^{th}$ and $(n-1)^{st}$ direction. This implies that $G(\sigma \times \{1\}) \cap_f H$ and so $g(\sigma) \cap_f H'$ as before. Since $g$ is $(m-1)$-sparse, the number of $(d-1)$-simplices $\sigma$ for which we have 
		
		$$G(\sigma \times [0,2]) \cap_f H$$
		
		\noindent is at most $S < S''$.

		Now suppose that $H$ is an $m$-plane that contains the $(n-1)^{st}$ direction, but does not contain the $n^{th}$ coordinate direction. In this case $H \subset \RR^{n-1} \times \{j\}$ for some $j \in \ZZ$. Observe that $G(\sigma \times [0,1])$ cannot have a full intersection with $H$, for any simplex $\sigma$, because $H$ does not contain the $n^{th}$ direction. Next, note that if $G(\sigma \times [1,2]) \cap_f H$, then $G(\sigma \times \{1\}) 
		\cap_f H$, since $H$ contains the $(n-1)^{st}$ direction. So by our condition on the labeling $L$, the number of $(d-1)$-simplices $\sigma$ for which we have 
		
		$$G(\sigma \times [0,2]) \cap_f H$$
		
		\noindent is a most $S'=(d+1)S''$, since each $\partial \sigma_i$ has $(d+1)$, $(d-1)$-simplices. Putting the two cases together we see that any $m$-plane has a full intersection with at most $S'$ of the $G(\sigma \times [0,2])$. Recall that since $g$ is $(m-1)$-sparse the image of each $(d-1)$-dimensional simplex $\sigma$ under $g$ has a full intersection with at most $S$ $(m-1)$-planes. It follows that the image of any $d$-dimensional polyhedron in $X \times [0,1]$, has a full intersection with at most $S$ $m$-planes. This is because $G(\partial \sigma_i \times [0,1])$ was defined as a cylindrical extension of the set $g(\sigma_i)$. Similarly, the image of any $d$-dimensional polyhedron in $X \times [1,2]$, has a full intersection with at most $S$ $m$-planes. This shows that $G$ is a $m$-sparse map with overlap $S'$. 
		
		Finally, we check that $G|_{X \times \{2\}}$ is $(m-1)$-sparse. Suppose $H$ is an $(m-1)$-plane in $\RR^{n-2} \times \{0\} \times \RR$ so that $H \cap_f G(X \times \{2\})$. Let $H'$ be the $m$-plane made from $H$ by adding in the $(n-1)^{st}$ coordinate direction. Since $G$ is $m$-sparse, for a fixed $H$, there are at most $S'$ $(d-1)$-simplices $\sigma$ in $X$, such that $H' \cap_f G(\sigma \times [1,2])$. And so there are at most $S'$ $(d-1)$-simplices $\sigma$ in $X$, such that $H \cap_f G(\sigma \times \{2\})$. Similarly, for a fixed $\sigma$, there are at most $S'$ $(m-1)$-planes $H$ in $\RR^{n-2} \times \{0\} \times \RR$, such that $H' \cap_f G(\sigma \times [1,2])$. And so there are at most $S'$ $(m-1)$-planes $H$ in $\RR^{n-2} \times \{0\} \times \RR$, such that $H \cap_f G(\sigma \times \{2\})$.  This shows that $G|_{X \times \{2\}}$ is $(m-1)$-sparse which finishes the proof of the lemma.
	\end{proof}
	
	Next we'll prove an extension lemma which will serve as the inductive step in the proof of \cref{sparse}.
	
	\begin{lemma} \label{extension} Fix three dimensions $1 \le d \le m < n$ and let $Y$ be a $d$-dimensional simplicial complex with at most $V$ vertices, so that each vertex lies in at most $D$ simplices. Suppose that for some $R>0$ we have a map,
		$$f: Y(d-1) \to [0, R]^{n-1} \times \{0\} \subset \RR^{n}$$
		
		\noindent which is $(m-1)$-sparse with overlap $S$. Then there is an extension of $f$,
		
		$$F: Y \to [0, R]^n  \subset \RR^{n}$$
		
		\noindent which is $m$-sparse with overlap $S'(n,D,S)$, where $S'(n,D,S)$ only depends on $n, D$ and $S$.
	\end{lemma}
	
	\begin{proof} Let $\sigma_1, \sigma_2 \ldots \sigma_N$ be the $d$-simplices of $Y$ and let $X = \bigsqcup_{i=1}^N \partial \sigma_i$. We will view $f$ as a map from, this is because each $\partial \sigma_i$ is a subset of both $X$ and of $Y$. Note that $f$ is $(m-1)$-sparse with overlap $DS$ when viewed as a map from $X$. Since the constant $S'(n,D,S)$ is allowed to depend on $S$ and $D$, it suffices to assume that $Y = \bigsqcup_{i=1}^N \sigma_i$ and $Y(d-1)$ is the $X$ just defined. This is because we have already fixed our map on the boundaries of these simplices. We will construct the map $F$ as a concatenation of $n-1$ maps. The  $i^{th}$ map will be $F_i: X \times [0,2] \to [0,R]^n$ and the maps will satisfy the following properties
		
		\begin{enumerate}
			\item $F_{n-1}$ extends $f$ and for each $i$, $F_i$ extends $F_{i-1}|_{X \times \{2\}}$. In other words, 
			$$F_i|_{X \times \{0\}} = F_{i-1}|_{X \times \{2\}}$$
			\item $F_i|_{X \times \{2\}}$ is $(i-1)$-tight and has image in some side of the cube $[0,R]^n$.
			\item $F_i$ is $m$-sparse with constant $S_i(n,D,S)$, which only depends on $n,D$ and $S$.
		\end{enumerate}
		
		We now explain how to define these maps. Define $F_{n-1}$ by using \cref{tight} to extend the given $(n-1)$-tight map $f$. The map $f$ is $(n-1)$-tight simply because its image is contained in a $(n-1)$-plane. \cref{tight} guarantees that $F_{n-1}|_{X \times \{2\}}$ is $(n-2)$-tight, and that $F_{n-1}$ is $m$-sparse with overlap $S_{n-1}(n,D,S)$ which only depends on $n,D$ and $S$. Next, we define $F_{n-2}$ by using \cref{tight} to extend the $(n-2)$-tight map $F_1|_{X \times \{2\}}$. Strictly speaking, to apply \cref{tight}, we need to rotate the coordinate system so that the image of $F_{n-1}|_{X \times \{2\}}$ lies inside $[0,R]^{n-1} \times \{0\}$. \cref{tight} guarantees that $F_{n-2}|_{X \times \{2\}}$ is $(n-3)$-tight, and that $F_{n-2}$ is $m$-sparse with overlap $S_{n-2}(n,D,S)$ which only depends on $n,D$ and $S_{n-1}(n,D,S)$. We can continue in this way to define each $F_i$. Specifically, to define $F_{i-1}$ apply \cref{tight} to the $(i-1)$-tight map $F_i|_{X \times \{2\}}$, after possibly rotating the coordinate system. By \cref{tight}, $F_{i-1}$ is $m$-sparse with overlap $S_{i-1}(n,D,S)$, which depends on $n,D$ an $S_i(n,D,S)$. But since $i < n$, we see that each constant $S_{i}(n,D,S)$ only depends on $n,D$ and $S$. Also notice that $F_{1}|_{X \times \{2\}}$ is $0$-tight. So for each $d$-simplex $\sigma$ in $Y$, $F_{n-1}(\partial \sigma \times \{2\})$ is a point in $\ZZ^n$. 
		
		We now concatenate the $F_i$'s to define $F$. This concatenation can be written out explicitly if we parametrize each $d$-simplex $\sigma$ in $Y$ as
		
		$$\sigma = \{(x, t): x \in \partial \sigma, t \in [0, n-1]\}/ \{(x, 0): x \in \partial \sigma\}$$
		
		\noindent For $1 \le i \le n-1$, $x \in \partial \sigma$ and $t \in [0,1]$ define
		
		$$F(x, i - t) = F_i(x, 2t)$$
		
		\noindent This map is well defined since $F(\partial \sigma \times \{0\}) = F_{1}(\partial \sigma \times \{2\})$ is a point in $\ZZ^n$. Since each $F_i$ was obtained using \cref{tight}, the image of $F$ lies in $[0,R]^n$ and $F$ is $m$-sparse with overlap $S'(n,D,S)$, which is some constant that depends only on $n, D$ and $S$.\footnote{We can be a bit more specific about the constant $S'(n,D,S)$. Let $\alpha$ be a tower of exponentials with base 2 and height $n$. Then $S'(n,D,S)$ is very roughly bounded by $(nDS)^{\alpha}$. This constant seems far from optimal, but we did not attempt to improve it in this paper.}
	\end{proof}
	
	\begin{proof}[Proof of \cref{sparse}] We define the map $F$ inductively by skeleta. Let $R = V^{\frac{1}{n-m}}$. By \cref{sparsevertex}, there is a map from the vertices of $Y$ to $[0,R]^{n-d} \times \{0\}^{d}$, which is $(m-d)$-sparse with overlap $S_1 = 2^{n-d}$. Define $F$ on $Y(0)$ to be this map. Now suppose for some $0 \le k < d$, we have already defined $F$ to map $Y(k)$ into $[0,R]^{n-d+k} \times \{0\}^{d-k}$ and $F$ is $(m-d+k)$-sparse with overlap $S_k$ on $Y(k)$, where $S_k$ only depends only on $n, D$ and $S_{k-1}$. By \cref{extension}, we can extend $F$ to map $Y(k+1)$ into  $[0,R]^{n-d+k+1} \times \{0\}^{d-k-1}$ so that $F$ is $(m-d+k+1)$-sparse on $Y(k+1)$ with overlap $S_{k+1}$, where $S_{k+1}$ depends only on $n, D$ and $S_{k}$. But since $k$ is at most $d$, each constant $S_{k}$ only depends on $n$ and $D$. Proceeding in this way we define the map $F: Y\to [0,R]^n$ which is $m$-sparse with overlap $S_d$, where $S_{d}$ only depends on $n$ and $D$.\footnote{Let $\beta$ be a tower of exponentials with base 2 and height $n^2$. The final sparseness constant $S_d$ is very roughly bounded by $(nD2^n)^{\beta}$.}
	\end{proof}
	
	In fact, the proof of \cref{sparse} yields the following relative extension theorem, which will be useful in the next section.
	
	\begin{theorem} \label{sparse_extension}  Fix three dimensions $0 \le d \le m < n$ and an integer $D > 0$. Let $Y$ be a simplicial complex with $V$ vertices and each vertex lies in at most $D$ simplices. Let $X$ be a subcomplex of $Y$ and suppose we have a map
		
		$$f: X \to [0,V^{\frac{1}{n-m}}]^n  \subset \RR^{n}$$
		
		\noindent so that for each $0 \le k \le d$, $f(X(k)) \subset [0,V^{\frac{1}{n-m}}]^{n-d+k} \times \{0\}^{d-k}$ and $f$ is $(m-d+k)$-sparse with overlap $S$ on $X(k)$. Then we can extend $f$ to a map
		
		$$F: Y \to [0, V^{\frac{1}{n-m}}]^{n} \subset \RR^{n}$$
		
		\noindent which is $m$-sparse map with overlap $S'(n,D,S)$, which only depends on $n, D$ and $S$.
	\end{theorem}
	
	\cref{sparse} says that a simplicial complex with $V$ vertices, and a small number of simplices adjacent to each vertex, admits a $m$-sparse map with small overlap into a cube of side-length $V^{\frac{1}{n-m}}$ in $\RR^n$. This estimate on the side-length is sharp because there exists a family of complicated simplicial complexes called high dimensional expanders. We now briefly introduce this family of simplicial complexes. Define the $k$-waist of a $d$-dimensional simplicial complex $Y$ as follows
	
	$$W_k(Y) =  \min_{f:Y \to \RR^k} \max_{p \in \RR^k} |f^{-1}(p)|$$
	
	\noindent where the min is over all continuous map to $\RR^k$, and $|f^{-1}(p)|$ stands for the number of $d$-dimensional simplices that intersect the fiber $f^{-1}(p)$.  If $d \ge 2$, a family of $d$-dimensional simplicial complexes $\{Y_i\}_{i=1}^{\infty}$ is a called a high dimensional expander if the following conditions are satisfied.\footnote{There are several different notions of high dimensional expanders in the literature. The one used in this paper is often referred to as topological high dimensional expanders.}
	
	\begin{enumerate}
		\item The number of vertices in $Y_i$ is $V_i$ and $\{V_i\}_{i=1}^{\infty}$ is unbounded.
		\item The number of simplices adjacent to any vertex in any $Y_i$ is uniformly bounded by some number $D$, independent of $i \ge 1$.
		\item For each $i \ge 1$ we have, $W_d(Y_i) \ge cV_i$ for some constant $c > 0$ independent of $i$. 
	\end{enumerate}
	
	\noindent In \cite{GE}, Gromov investigated various waist properties of simplicial complexes and asked whether high dimensional expanders exist. In \cite{EK}, Evra and Kaufman remarkably gave an example of a family of high dimensional expanders in any dimension $d$. 
	
	Let $\{Y_i\}$ be our family of high dimensional expanders. Suppose for each $i \ge 1$, there was an $R_i > 0$ and a map $F_i: Y_i \to [0,R_i]^n \subset \RR^n$ which was $m$-sparse with overlap $S$. Let $\pi: \RR^n \to \RR^d$ be the orthogonal projection onto some $d$-dimensional subspace in the unit lattice. By the definition of a high dimensional expander, there must be some $p_i \in \RR^d$ so that $(\pi \circ F_i)^{-1}(p_i)$ intersects at least $cV_i$ $d$-simplices in $Y_i$. On the other hand, $\pi^{-1}(p_i)$ can intersect at most ${n \choose m}R_i^{n-m}$, $m$-planes inside $[0,R_i]^n$. Observe that if $(\pi \circ F_i)^{-1}(p_i)$ intersects some $d$-simplex $\sigma$ (in any way), then there is some $m$-plane $H$ such that $H \cap_f F_i(\sigma)$  and which intersects $\pi^{-1}(p_i)$. Because $F_i$ is $m$-sparse, $(\pi \circ F_i)^{-1}(p_i)$ can intersect at most $S{n \choose m}R_i^{n-m}$ $d$-simplices from $Y_i$. Putting these observations together we get
	
	$$cV_i \le S{n \choose m}R_i^{n-m}$$
	
	\noindent This shows that \cref{sparse} is sharp, up to a constant depending on $n$ and $D$, for high-dimensional expanders.
	
	To end this section we show that sparse maps of large codimension can me perturbed to satisfy the main constraint that appears in \cite{GG}, that is, the maps separate disjoint simplices. We remark that the maps we construct need not be an embedding on each simplex, but with a little bit more work they can probably be made into an embedding.
	
	\begin{theorem} \label{thick} Fix three dimensions $0 \le d < n$ so that $2d+1 \le n$ and fix an integer $D > 0$. Suppose $Y$ is a $d$-dimensional simplicial complex with $V$ vertices, so that each vertex lies in at most $D$ simplices. Then there is an map
		
	$$F: Y \to [0, V^{\frac{1}{n-d}}]^{n} \subset \mathbb{R}^{n}$$
	
	\noindent which satisfies the following condition: if $\sigma$ and $\tau$ are two $d$-simplices of $Y$ which do not share any vertices then 
	
	$$dist_{\RR^n}(F(\sigma), F(\tau)) \ge c(n,D)$$
	
	\noindent where $c(n,D)$ is some constant only depending on $n$ and $D$.
	\end{theorem}

	\begin{proof} By \cref{sparse}, we know that there is a $d$-sparse map $F': Y \to [0, V^{\frac{1}{n-d}}]^{n}$ with over lap $S = S(n,D)$. The proof of this theorem is just an application of the techniques from \cite{GG} applied to the map $F'$. We first subdivide the image of $F'$, and then randomly perturb this subdivision to a map $F$ with the desired property. Notice that the image of each simplex under $F'$ looks like a collection of cells (of possibly different dimensions) from the lattice $\Lambda^n$. This is evident from the construction of the maps in \cref{tight}, which are made from several projections. However, the image of each simplex need not be injective and can decrease dimensions. Let $Y'$ be the space obtained from $Y$ by taking the Stein factorization of each simplex. That is,
		
	$$Y' = Y / \{x \sim y : F'(x) = F'(y), x \in \sigma, y \in \sigma \text{ for some simplex $\sigma$ in Y}\} $$
	
	The map $F'$ evidently factors through $Y'$. Now notice that $Y'$ admits a natural cubical structure which is obtained by intersecting the image of each simplex of $Y$ with the lattice $\Lambda^n$. Next, we can subdivide this cubical structure to a simplicial structure, by subdividing each cube into at most $C$ simplices, for some constant $C$ only depending on $n$. From now on we will view $Y'$ as a simplicial complex with this structure. Next we will perturb the map $F'$, viewed as a map from $Y'$. Because $F'$ factors through $Y'$, this will induce a perturbation of $F'$ viewed as a map from $Y$. Notice that since $Y'$ is a simplicial complex, we can define the perturbation on the vertices and then extend the map linearly to the simplices. For each vertex $v \in Y$ choose a random vector $\rho(v)$ in a ball of radius $1$. Define the perturbation $F$ on $v \in Y'$ to be $v + \rho(v)$, and then extend $F$ to $Y'$ linearly.
	
	For any two simplices of $Y'$ (of any dimension) which do not share any vertices, $\sigma$ and $\tau$, let $Bad_{\eps}(\sigma, \tau)$ be the bad event that $dist_{\RR^3}(\sigma, \tau) \le \eps$. Proposition 3.12 from \cite{GG} tells us that the probability of $Bad_{\eps}(\sigma, \tau)$ is at most $C\eps$ for some constant $C$ only depending on $n$. The proof of this proposition relies on a fact about determinants of random matrices with bounded entries. Notice that if $dist_{\RR^3}(F'(\sigma), F'(\tau)) > 3$, then we also have $dist_{\RR^3}(F(\sigma), F(\tau)) > 1$, because the perturbation moves each vertex a distance of at most 1. Because $F'$ is $d$-sparse, each ball of radius 3 intersects at most $S'$ simplices of $Y'$, where $S'$ is some constant that depends only on $n$ and $S$. Thus, each bad event is not independent of at most $A$ other bad events, where $A$ depends only on $n$ and $S$. By the Lovasz-Local lemma, we can ensure that no bad event happens with positive probability if $\eps$ is small enough, depending only on $A$. If no bad events happen, then the images under $F$ of any two $d$-simplices of $Y$, which do not share any vertices, are $\eps$-separated and so the condition in the statement of the theorem is satisfied for $c(n,D) = \eps$.

	\end{proof}

	%%%
	\section*{\centering Maps from simplicial complexes with controlled 1-waist}
	
	In this section we prove \cref{smallwaist}.
	
	\begin{theorem*} [2] Suppose $Y$ is a $d$-dimensional simplicial complex with $V$ vertices, so that each vertex lies in at most $D$ simplices. Then for $n \ge d+1$, there is a map
		
		$$F: Y \to [0, V^{\frac{1}{n}} W_1(Y)^{\frac{d}{n(n-d)}}]^n \subset \RR^n$$
		
		\noindent so that the pre-image of each unit ball intersects at most $C(n, D)$ simplices, where $C(n, D)$ only depends on $n$ and $D$.
	\end{theorem*} 
	
	We first prove a small technical lemma.
	
	\begin{lemma} There is a map $f: Y \to [0,1]$ which is linear on each simplex of $Y$, injective on the vertices of $Y$, and each fiber of $f$ intersects at most $3W_1(Y)$ $d$-simplices of $Y$.
	\end{lemma}
	
	\begin{proof} Note that if we have a map from $Y$ to some compact interval, then scaling the image does not change the size of the fibers. Thus, we can assume without loss of generality that we have a map $f_1: Y \to [0,1]$ which realizes the 1-waist of $Y$, that is, each fiber of $f_1$ intersects at most $W_1(Y)$ $d$-simplices of $Y$. Define $f_2: Y \to [0,1]$ be the map which agrees with $f_1$ on the vertices of $Y$, and is linear on each simplex of $Y$. Notice that for any simplex $\sigma$ in $Y$, $f_2(\sigma) \subset f_1(\sigma)$. So each fiber of $f_2$ also intersects at most $W_1(Y)$ $d$-simplices of $Y$. Now define $f: Y \to [0,1]$ to be a perturbation of $f_2$, so that it is injective on the vertices, but still linear on each simplex of $Y$. If the perturbation is small enough, then for any $p \in \RR$ and small enough $\epsilon >0$, we get $|f^{-1}(p)| \le |f_2^{-1}(p-\epsilon)| + |f_2^{-1}(p)| + |f_2^{-1}(p + \epsilon)|$. So each fiber of $f$ intersects at most $3W_1(Y)$ $d$-simplices of $Y$. 
	\end{proof}
	
	\begin{proof}[Proof of \cref{smallwaist}] Let $f: Y \to [0,1]$ be the map from the previous lemma. Set $W = 3W_1(Y)$. Since $f$ is injective on the vertices of $Y$, we can choose points $0 = p_0  < p_1 < p_2 \ldots p_{N-1} < p_N=1$, so that for each $0 \le i < N$, $f^{-1}(p_i)$ does not contain any vertices of $Y$ and $f^{-1}((p_i, p_{i+1}])$ contains between $W$ and $W+D$ vertices of $Y$. Note that $N$ is equal to $VW^{-1}$ times some constant depending on $d$ and $D$. Since $f$ is linear on the simplices and injective on the vertices, each fiber $f^{-1}(p_i)$ can intersect a $d$-simplex of $Y$ in a polyhedron of dimension $(d-1)$ with at most $C(d)$ sides, where $C(d)$ is some constant that only depends on $d$. By the definition of the 1-waist, the fiber $f^{-1}(p_i)$ consists of at most $W$ such polyhedra. So, there is a subdividision of $Y$ that contains each $f^{-1}(p_i)$ in its $(d-1)$-skeleton, contains at most $C(d)V$ vertices, and each vertex in the subdivision is contained in at most $C(d)D$ simplices. Let $Y'$ be this subdivision of $Y$. It now suffices to find a map from $Y'$ to $\RR^n$ that satisfies the conditions of the theorem. 
		
		We begin by finding a map $F'$ from $Y'$ into a long rectangular prism in $\RR^n$. This rectangular prism will be divided into $N$ cubes, and $F'$ will map the subcomplex  $f^{-1}([p_i, p_{i+1}]) \subset Y'$ near the $i^{th}$ cube using a $d$-sparse map. For $r,t>0$, let $Q_r(t)$ denote the $n$-dimensional cube with side-length $r$ centered at the point $(t+r/2, r/2, r/2 \ldots r/2) \in \RR^n$. Notice that the bottom face of such a cube lies in $\RR^{n-1} \times \{0\}$. Fix $R$ to be the largest number of vertices in any subcomplex $f^{-1}([p_i, p_{i+1}])$ raised to the power of $\frac{1}{n-d}$. So $R$ is at most $W^{\frac{1}{n-d}}$ times a constant that depends only on $d$ and $D$. By \cref{sparse}, for each $i$ there is a map from the subcomplex $f^{-1}(p_i) \subset Y'$ to $Q_{R}(iR)$ which is $d$-sparse with overlap $S(n,D)$. By the proof of \cref{sparse} this map is $k$-sparse on the $k$-skeleton of $f^{-1}(p_i)$, and maps this $k$-skeleton into 
		
		$$Q_{R}(iR) \cap \RR^{n-d+k} \times \{0\}^{d-k}$$
		
		That is, it maps it into the the $(n-d+k)$-dimensional face of the $i^{th}$ cube. For every $i$, define $F'$ on $f^{-1}(p_i)$ to be this map. By \cref{sparse_extension}, for each $i$ there is a map that extends the map just defined on $f^{-1}(p_i) \cup f^{-1}(p_{i+1})$, to a map from $f^{-1}([p_i, p_{i+1}])$ into $Q_{2R}(iR)$, and this map is $d$-sparse with overlap $S'(n,D)$. For each $i$, define $F'$ on $f^{-1}([p_i, p_{i+1}])$ to be this map. If we define the following rectangular prism
		
		$$P = \bigcup_{i=1}^N  Q_{2R}(iR)$$
		
		\noindent then we see that $F'$ has been defined to map $Y'$ into $P$. As defined, $F'|_{f^{-1}([p_i, p_{i+1}])}$ is $d$-sparse with overlap $S'(n,D)$ and has image in the $i^{th}$ cube. So the $F'$-preimage of any unit ball intersects at most $S''(n,D)$ simplices, where $S''(n,D)$ only depends on $n$ and $D$.
		
		Recall that $R$ is roughly $W^{\frac{1}{n-d}}$ and that $N$ is roughly $VW^{-1}$. There are $N$ cubes in our rectangular prism $P$, so the volume of $P$ is $VW^{\frac{d}{(n-d)}}$ times some constant depending only on $n$ and $D$. Also one side of $P$ is roughly $N$ times longer than all the others, and we can assume that $N$ is some number much larger than $1$. So there is an injective map that snakes $P$ into a cube of roughly the same volume as $P$, which is $B(n, D)$-bilipschitz on every unit ball, where $B(n, D)$ is a constant that only depends on $n$ and $D$. Let this map be
		
		$$I: P \to [0, V^{\frac{1}{n}}W^{\frac{d}{n(n-d)}}]^n \subset \RR^n$$
		
		\noindent Define our final map to be $F = I \circ F': Y \to \RR^n$. By the properties of $I$ and $F'$, and since $Y'$ is a subdivision of $Y$, the number of simplices of $Y$ that can intersect any $F$ pre-image of a unit ball is some constant that only depends on $B(n,D), S''(n,D), n$ and $D$. In particular, it only depends on $n$ and $D$, and so $F$ satisfies the condition in the statement of the theorem.
	\end{proof}

\end{document}